\documentclass[11pt]{article}

\usepackage{graphicx,hyperref,cite}
\usepackage[greek,russian,english]{babel}
\usepackage[utf8x]{inputenc}

\newcommand{\gr}[1]{\foreignlanguage{greek}{#1}}

\newtheorem{theorem}{Theorem}
\newtheorem{lemma}{Lemma}

\newtheorem{observation}{Observation}

\newtheorem{corollary}{Corollary}

\newenvironment{proof}{\noindent {\bf Proof.}}{$\Box$\medskip}

\usepackage{amsmath,amssymb,amsfonts,graphics,color,cite}

\newcommand{\noindents}[2]{\vspace{-#1mm}\noindent {#2}}
\newcommand{\remove}[1]{}

\newcommand{\s}{\mbox{$\mathbb{S}_{0}$}}

\newcommand{\wbw}{{\bf wbw}}

\newcommand{\wxy}[1]{{\bf wbw}\mbox{$_{#1}(x,y)$}}
\newcommand{\tria}{_{\mbox{\tiny $\Delta$}}}

\newcommand{\ourxi}{\mbox{\gr{\bf ξ}}}
\newcommand{\ourlambda}{\mbox{\gr{\bf λ}}}
\newcommand{\smallourlambda}{\mbox{\gr{\bf\footnotesize λ}}}

 \usepackage{titling}
\begin{document}

\title{Nearly Planar Graphs and $\lambda$-flat Graphs\thanks{The work of the two last authors 
was {co-financed by the European Union (European Social Fund - ESF) and Greek national funds through the Operational Program ``Education and Lifelong Learning'' of the National Strategic Reference Framework (NSRF) - Research Funding Program: ``Thales. Investing in knowledge society through the European Social Fund''.}} $^{,}$\thanks{{Emails: 
\mbox{\sf a.grigoriev@maastrichtuniversity.nl},
\mbox{\sf akoutson@math.uoa.gr},
\mbox{\sf sedthilk@thilikos.info}}}}
\author{Alexander Grigoriev\thanks{School of Business and Economics
Department of Quantitative Economics, Maastricht University}\ \ , Athanassios Koutsonas\thanks{Department of Mathematics, National and Kapodistrian University of Athens, Athens, Greece.}\ \ ,  Dimitrios M. Thilikos$^{\rm \S,}$\thanks{AlGCo project-team, CNRS, LIRMM, France.}}
\date{\empty}

\maketitle
\vspace{-14mm}
\begin{abstract}

\noindent A graph $G$ is {\em $\xi$-nearly planar} if it can be embedded in the sphere 
so that each of its edges is crossed at most $\xi$ times.
The family of $\xi$-nearly planar graphs is widely extending the
notion of planarity.
We introduce an alternative parameterized graph family extending the notion of planarity, the  $\lambda${\em-flat graphs}, this time defined as powers of plane graphs in regard to a novel notion of distance, the {\em wall-by-wall distance}. We show that the two parameterized graph classes are parametrically equivalent.
\noindent \end{abstract}


\section{Introduction}
\label{intro}

The class of {\em  $\xi$-nearly planar graphs} 
contains all graphs  embeddable in the plane 
with at most $\xi$ crossings per edge (see Figure~\ref{fig:this2one} for an example of a $2$-nearly planar graph). This graph family was introduced for the first time by Pach and Tóth~\cite{PachT97grap} 
under the name {\sl $\xi$-quasi planar graphs} (we adopt  the term {\em  $\xi$-nearly planar} as the term  {\em $\xi$-quasi planar} has been extensively 
used to denote embedded graphs without $\xi$ pairwise crossing edges, see~\cite{FoxPS11then,Suk06kqua,AgarwalAPPS97quas,AckermanT07note}).

\begin{figure}[ht]
\begin{center}
	\scalebox{.4}{\includegraphics{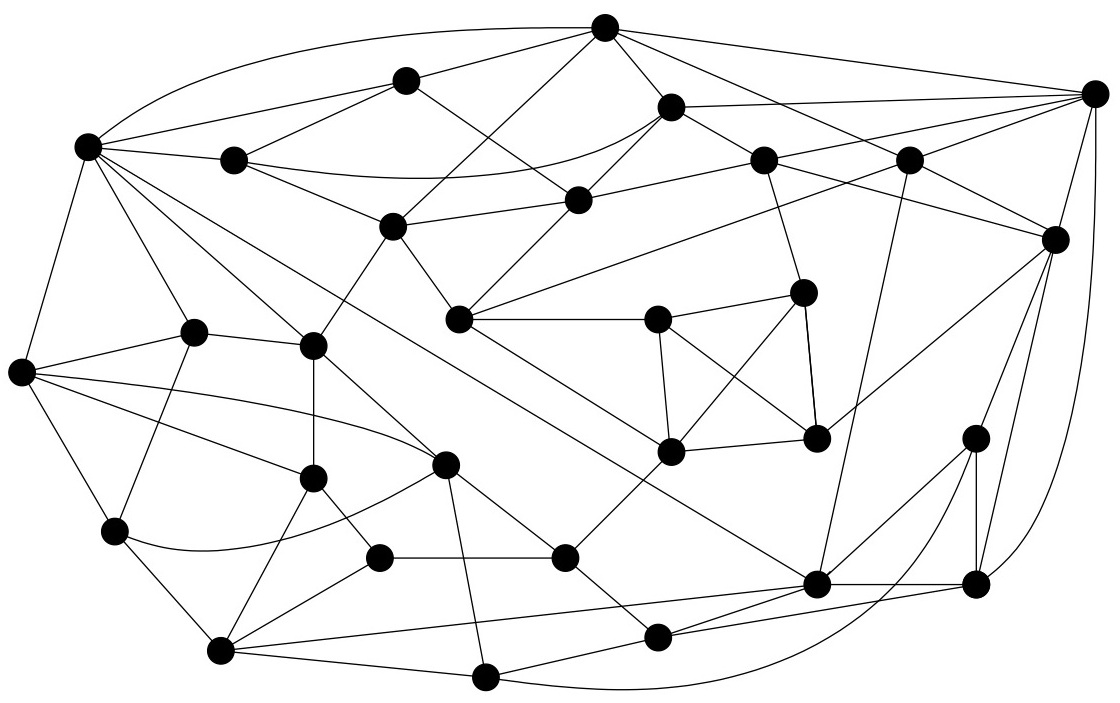}}
\end{center}
\caption{\small A $2$-nearly planar graph.}
\label{fig:this2one}
\end{figure}

The class of {$\xi$-nearly planar graphs} does not have  the good closedness properties of topologically-defined graph classes. In particular, they are not  closed under topological minors or contractions -- see Corollary~\ref{notclosed} -- however  they are closed  under subgraphs.
We stress that the study of nearly planar graphs marks a direction that is ``orthogonal'' to the one  
of (topological) minor free graphs as every graph is a topological minor of some
$1$-nearly planar graph:  take any drawing of a graph in the plane  
and for each edge $e$, introduce a vertex in each segment delimited by two consecutive crossings
along $e$.  

Several combinatorial  and algorithmic results on planar graphs 
have been extended to $\xi$-nearly planar graphs. For instance,
according Pach and Tóth~\cite{PachT97grap}, 
$\xi$-nearly planar graphs are sparse: 
each $n$-vertex $\xi$-nearly planar graph  has at most $4.108\cdot \sqrt{\xi}\cdot n$ edges.
Moreover Bodlaender and Grigoriev proved that several 
problems in graphs admit EPTAS (Efficient Polynomial Time Approximation Schemes)
when restricted to 
$\xi$-nearly planar graphs~\cite{GrigorievB07algo}.\medskip%

Consider a distance metric between the vertices of graphs.
We define the {\em $\lambda$-power} of a graph $G$ under this metric
as the graph with the same vertex set as $G$ and 
where two vertices $x$ and $y$ are adjacent if and only if their distance 
is upper bounded by $\lambda$. 
Our aim  is to define such a metric  
that ``captures'' the concept of $\xi$-nearly planar graphs.
This metric is defined on embedded graphs and  is called 
{\sl wall-by-wall} distance 
%
(for the definition  
see Section~\ref{sec:classes}).

To explain how  wall-by-wall expresses $\xi$-nearly planar graphs, we 
 introduce the concept of {\em $\lambda$-flat graphs} that are 
subgraphs of the $\lambda$-powers of plane graphs under the new metric. 
Our main result is a proof,  that the classes of $\lambda$-flat graphs
and $\xi$-nearly planar graphs are {\em parametrically equivalent}, i.e.,
there are functions $g$ and $h$ such that 
every $\xi$-nearly planar graph is $g(\xi)$-flat and every $\lambda$-flat 
graph is $h(\lambda)$-nearly planar.
This enables us to represent every nearly planar graph $G$ by a suitably-defined 
plane graph $H$ whose structure 
``geometrically represents'' $G$. 

%
%

The paper is organized as follows: In Section~\ref{defprel},
we give some basic definitions and results.
In Section~\ref{sec:classes} we introduce the notion of $\lambda$-flat 
graphs and we prove several properties of them.
The proof of our main result is presented in Section~\ref{parma}.

\section{Definitions and preliminaries}
\label{defprel}

All graphs in this paper are undirected, without loops, and  may 
have multiple edges.
If a graph has no multiple edges or loops we call it {\em simple}.
Given a graph $G$, we denote by $V(G)$ its vertex set and by $E(G)$ its edge set. 
We say that $H$ is a {\em subgraph of $G$ via
an injection $\tau: V(H)\rightarrow V(G)$}
if for every edge $\{x,y\}\in E(H)$, $\{\tau(x),\tau(y)\}$ is an edge of $G$
and we denote this fact by $H\subseteq_{\tau} G$. We also say that $H$ is a subgraph of $G$
if $H\subseteq_{\tau} G$ via such an injection $\tau$ and we write $H\subseteq G$.
For any set of vertices $S\subseteq V(G)$, we denote by $G[S]$ the subgraph of $G$ induced by the vertices from $S$. 
Moreover, we denote by $N_{G}(S)$ the neighbors of the vertices in $S$ that do not belong to $S$.
We also set $N_{G}[S]=S\cup N_{G}(S)$ and $\partial_{G}(S)=N_{G}(V(G)\setminus S)$.

Given two graphs $H$ and $G$, we write $H\preceq G$ and call $H$ a {\em minor} of $G$, if $H$ can be obtained from a subgraph of $G$ by edge contractions
(the {\em contraction} of an edge $e=\{x,y\}$  in a graph $G$
 is the operation of replacing $x$ and $y$ by a new vertex $x_{e}$ that is made adjacent
with all the neighbors of $x$ and $y$ in $G$ that are different from $x$ and $y$).
Moreover, we say that $H$ is a {\em contraction} of $G$
if $H$ can be obtained from $G$ by contracting edges.

We say that $H$ is a {\em topological minor} of $G$ if 
a subdivision of $H$ is a subgraph of $G$ 
(the {\em subdivision} of an edge $e=\{x,y\}$ is its replacement by a path of length 2 whose 
endpoints are $x$ and $y$). 

To simplify notations, we use  “$O_{k}(n^c)$” instead of saying that “$f(k) \cdot  n^c$ for some recursive function $f : \Bbb{N} → \Bbb{N}$”. 

\noindents{-2}{\bf Plane graphs}.
Throughout this paper, we use the term {\em graph embedding} to denote 
an embedding of a graph in the sphere \s\ where
if two edges cross, they  cross to a point that is not a vertex of the graph.
We use the term {\em plane graph} for an embedding of a graph 
without crossings. 
A graph is {\em planar} if it has some planar embedding.
Given a plane graph $G$, we denote by $F(G)$ the set of faces of $G$, i.e.,
the connected components of $\s\setminus G$, that are open sets. 

\noindents{-2}{\bf Radial and medial graph}.
Given a plane graph $G$ with at least one edge,  we define its {\em radial graph} $R_{G}$ as the plane graph whose vertex set is $V(G)\cup V(G^{*})$ ($G^{*}$ is the dual graph of $G$) and whose
edges are defined as follows: let ${\cal C}=\{C_{1},\ldots,C_{r}\}$ be the connected components of $\s\setminus (G\cup G^{*})$ isomorphic to open discs. Observe
that for any $i=1,\ldots,r$, the component $C_{i}$ is an open set whose boundary is incident to one vertex $v_{i}\in V(G)$ and to one vertex $u_{i}\in V(G^{*})$.
The edge set of $R_{G}$ is the set $E(R_{G})=\{\{v_{i},u_{i}\}:\ i=1,\ldots,r\}$ where edge $\{v_{i},u_{i}\}$ has multiplicity $1$ if both, $v_{i}$ and $u_{i}$,
have degree at least $2$ in $G$ and $G^{*}$, respectively, otherwise, the edge multiplicity of $\{v_{i},u_{i}\}$ is $2$. Notice that $R_{G}=(V(G)\cup F(G),E(R_{G}))$ is a bipartite
graph, whose parts are the vertex and face sets of $G$, respectively.

Given a plane graph $G$, we define its {\em medial graph} $M_{G}$  as the dual of its radial graph, i.e.,
$M_{G}=R_{G}^{*}$. Clearly, the vertices of $G$ correspond to faces of $M_{G}$.

\noindents{-2}{\bf Plane transformation}.
We describe a procedure we will be using in order to transform any simple embedding of a (not necessarily planar) graph $G$ on the sphere, to a plane graph $G_{p}$. 
\begin{figure}[h]
\begin{center}
	\includegraphics[width= 4.9in]{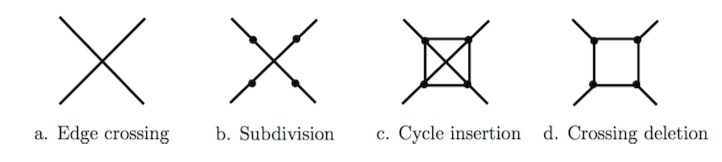}
	 \end{center}
	\caption{\small The four steps of the transformation of $G$ to $G_{p}$.}\label{fig:gadget} 
 \end{figure}
%
\begin{itemize}
\setlength{\itemsep}{-.0pt}
\item[(a)] modify the embedding of $G$ so that 
no more than two edges are crossing at the same point. 

\item[(b)] Subdivide the edges of $G$ and draw the resulting graphs such 
that each edge is crossed only once. 
\end{itemize}

\noindent \ \ \ In the resulting graph, for every crossed pair of edges $e=\{x,y\}$ and $e'=\{u,v\}$
\begin{itemize}
\setlength{\itemsep}{-.0pt}

\item[(c)] add the 
cycle $xuyv$ (notice that this cycle can be drawn without creating further crossings),
and 

\item[(d)] remove every such pair of crossing edges $e$ and $e'$.
\end{itemize}

\noindent Notice that, by construction,  $G_p$ is indeed a plane graph.
For an illustration of the construction of $G_p$,
see Figure~\ref{fig:gadget}. We will be referring to this procedure as {\em plane transformation} of a graph.

\section{The classes of $\lambda$-flat and $\xi$-nearly planar graphs.}
\label{sec:classes}

Let $G$ be a  plane graph. 
For any vertex $v\in V(G)$, the edges incident to $v$ appear in the drawing in the natural cyclic order which we denote  by ${\rm \sigma}_v$ (such a cyclic order is well defined as $G$ is loopless). 
Two edges are called {\em attached} if they share a vertex, say $v\in V(G)$, and appear consecutively in the cyclic order ${\rm \sigma}_v$ or its inverse.

\noindents{-2}{\bf Wall by wall distance}.
Let $e',e''\in E(G)$ be two edges in $G$. An $(e',e'')$-{\em wall-by-wall walk} 
is a sequence of distinct edges $e_1,e_2,\ldots, e_\ell$ such that
$e_1=e'$ and $e_\ell=e''$ and every two consecutive edges in $e_1,e_2,\ldots, e_\ell$ are attached. Intuitively, such a walk can be explained in the following way. Let the plane graph represent a building map where edges represent walls, each having a door. Then, a \emph{wall-by-wall-walk}, denoted also as a \wbw-walk, is a walk from wall $e'$ to wall $e''$ continuously touching walls and possibly going through doors.

Now, let $x$ and $y$ be two vertices of the plane graph $G$. The \wbw-walk between $x$ and $y$, also called {\em $(x,y)$-\wbw-walk}, is an $(e_x,e_y)$-\wbw-walk
such that the edges  $e_x$ and $e_y$ are incident to the vertices $x$ and $y$, respectively.
The {\em wall-by-wall distance between $x$ and $y$}, denoted by \wxy{G}, 
is the number of edges in the shortest $(x,y)$-\wbw-walk, if such a walk exists, otherwise it is 
infinite (in the case that no path in $G$ connects $x$ and $y$).
It is easy to verify that if $\wxy{G}\leq k$ for some integer $k$, then there exists a {\em wall-by-wall path} of length $k$, 
namely a \wbw-walk whose edges form an acyclic subgraph of $G$. 
Again, we use the notations {\em \wbw-path} and {\em $(x,y)$-\wbw-path}.

Notice that the wall-by-wall distance is lower bounded by the standard distance in graphs. However, the difference between the two distances may be arbitrarily large, [e.g. the distance of two non adjacent vertices of a star].
The following observation gives an easy way to express the wall by wall distance in terms of the common distance.

\begin{observation}
\label{obs:tdhisplo}
Let $G$ be a plane graph and let $M_{G}$ be its medial.
Let also $x$ and $y$ be two vertices in $G$ and let $f_{x}$ and $f_{y}$ be the faces
of $M_{G}$ corresponding to $x$ and $y$ respectively. 
Then the wall-by-wall distance between $x$ and $y$ in $G$ is one more than the 
minimum distance in $M_{G}$ between a vertex in the boundary of $f_{x}$ and a vertex in the boundary in $f_{y}$.
\end{observation}

\noindents{-2}{\bf $\lambda$-flat graphs}.
We consider taking powers of a plane graph with respect to the new metric. Let $H$ be a simple plane graph and let $\lambda$ be a positive integer. For $\lambda\geq 1$, 
we define
graph $H^ \lambda$ as a graph with vertex set $V(H)$ and edge set $E(H)\cup \tilde{E}$, where 
$\tilde{E}$ contains all pairs $\{x,y\}$ that do not belong to $E(H)$ and 
$\wbw_{H}(x,y)\leq \lambda$.
We say that a graph $G$ is {\em $\lambda$-flat} if it is a subgraph of  $H^\lambda$  for some planar graph $H$. Notice that, according to the above definition, $1$-flat graphs are exactly the planar graphs.  

\begin{figure}[ht]
\begin{center}
\scalebox{.32}{\includegraphics{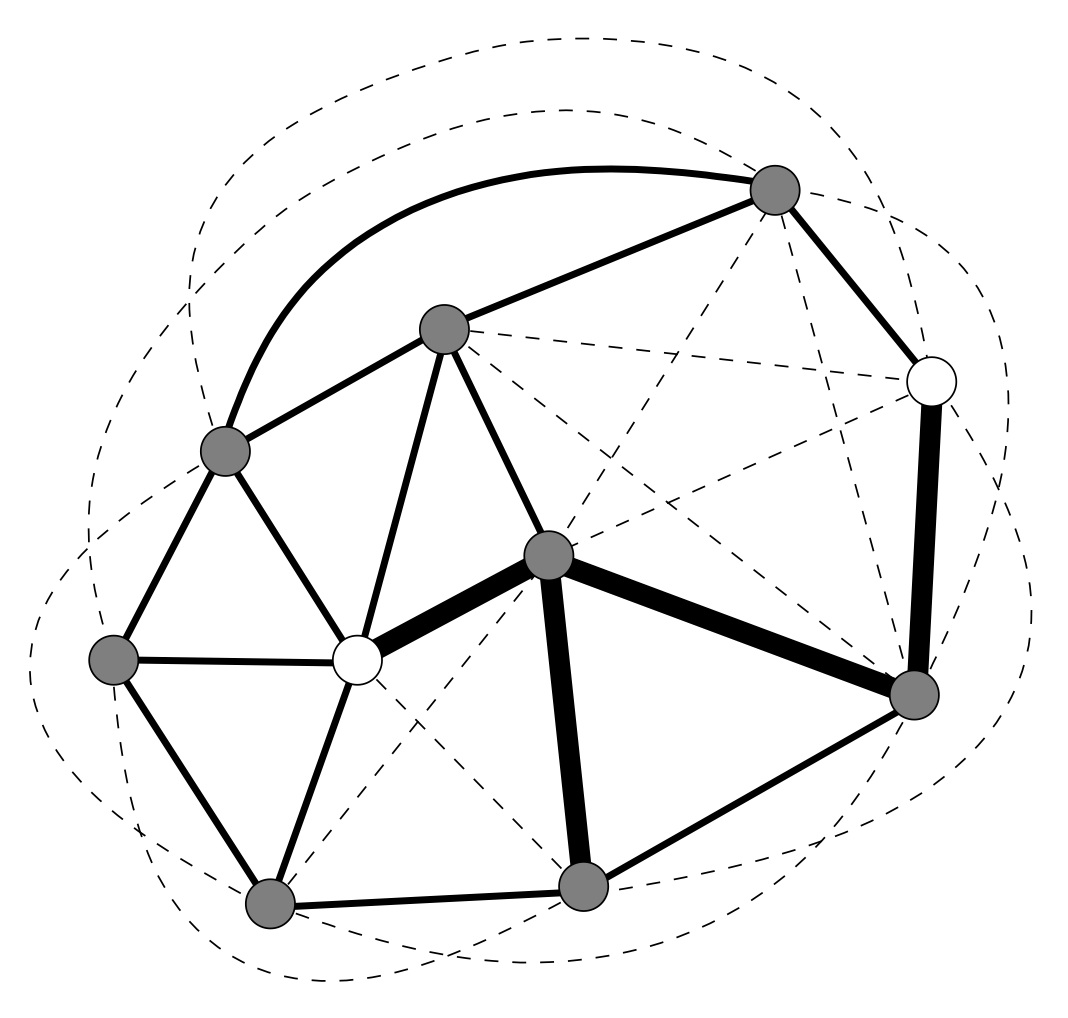}}
\end{center}\vspace{-3mm}
\caption{\small Graphs $G$ and $G^2$ (the dotted edges are the edges of $G^2$ that are not in $G$).
Notice that every \wbw-path between the two white vertices in $G$ has length at least 4 (one of these paths is depicted by the bold edges).}
\label{fig:thiswbw}
\end{figure}

\noindents{-2}{\bf $\xi$-nearly planar graphs}. 
We use the term {\em $\xi$-nearly planar graph embedding} for an embedding with at most  $\xi$ crossings per edge.
A graph $G$ is {\em $\xi$-nearly planar} if it has a $\xi$-nearly planar graph embedding.
For simplicity, we use the same notation for both $\xi$-nearly planar graphs and their embeddings.
In contrary to the case of planar graphs, even the recognition
of $1$-nearly planar graphs is a {\sf NP}-hard problem~\cite{GrigorievB07algo}.
The following theorem states our main result, which links the two aforementioned graph classes.

%

\begin{theorem}
\label{equiv}
The classes of $\lambda$-flat graphs and $\xi$-nearly planar graphs are parametrically equivalent.
\end{theorem}

From~\cite{PachT97grap}, each $n$-vertex $\xi$-nearly planar graph  has at most $4.108\cdot \sqrt{\xi}\cdot n$ edges, i.e. is sparse. An immediate consequence of 
Theorem~\ref{equiv} is that $\lambda$-flat graphs are also sparse.

\begin{corollary}
\label{smcor}
Every $\lambda$-flat graph on $n$ vertices has ${O_\lambda}(n)$ edges.
\end{corollary}

By definition, $\lambda$-flat graphs are closed under 
subgraphs (i.e. a subgraph of a $\lambda$-flat graph is also $\lambda$-flat).  
We prove that this is not the case for  topological minors and 
contractions.

\begin{corollary}
\label{notclosed}
For every $\lambda\geq 2$, $\lambda$-flat graphs are not closed neither under taking topological minors nor under taking contractions. The same holds also for $\xi$-nearly planar graphs for $\xi\geq 1$.
\end{corollary}

\begin{proof}
We present the proof for the case of $\lambda$-flat graphs. The
proof for $\xi$-nearly planar graphs is very similar and uses the same gadgets.

Let $K_{n}$ be a complete graph that is not $\lambda$-flat, for some integer $\lambda$
(such a value for $n$ exists because of Corollary~\ref{smcor}).
We will show that a $2$-flat graph $G_{n}$ contains the clique $K_{n}$
as a topological minor/contraction.

For the case of the topological minor relation, we construct $G_{n}$ as follows.
Let $H_{n}$ be the plane transformation of an embedding of $K_{n}$.
Let also $G_{n}$ be the graph obtained after the third
step of the plane transformation of $K_{n}$ to $H_{n}$, namely before eliminating the pair of crossed edges.
By construction, the  graph $G_{n}$ contains $K_n$ as a topological minor. We now claim that $G_{n}$ is $2$-flat. To see this, consider the plane graph $H_{n}$.
For any pair of deleted crossed edges $(x,y)$ and $(u,v)$, $H_{n}$ has a face 
bounded by the cycle $xuyv$. This implies that
\wxy{H_{n}}$=2$ and \wbw$_{H_{n}}(u,v)=2$. Thus $H_{n}^2$ contains all deleted edges. This implies that 
$G_{n}\subseteq  H_{n}^{2}$ and therefore $G_{n}$ is $2$-flat.

For the case of the contraction relation, we construct $G_{n}$ as follows.
Let $H_{n}$ be a $(2n\times 2n)$-grid and let $G_{n}=H_{n}^{2}$.
Certainly, $G_{n}$ is a $2$-flat graph. 
Moreover, it is easy to verify that $H_{n}$ can be contracted to the complete bipartite $K_{n,n}$
and this in turn to a clique of $n$ vertices. Therefore the same holds for $G_n$.
\end{proof}

By Corollary~\ref{smcor}, there is no $\lambda$ such that all graphs are $\lambda$-flat. 
We define the following two functions from graphs to integers 
\begin{align}
\ourlambda(G) &= \min\{\lambda\mid G \mbox{\ is $\lambda$-flat}\} \text{ and}\notag\\
\ourxi(G) &= \min\{\xi\mid \mbox{$G$ is $\xi$-nearly planar}\}.\notag
\end{align}
According to Theorem~\ref{equiv}, the two functions  $\ourlambda$ 
and $\ourxi$ are parametrically equivalent, i.e., the one is bounded by a function of the other. This means that the {\sl wall-by-wall} metric permits us to treat  $\xi$-nearly planar graph embeddings
as subgraphs of powers of plane graphs. 

\section{Proof of equivalence}
\label{parma}

%
This section is dedicated to the proof of Theorem~\ref{equiv}.


%



\begin{lemma}
\label{cross:flat}
Let $G$ be a simple $\xi$-nearly plane graph on $n$ vertices, Then, there exists a
graph $H$ on $O_{\xi}(n)$ vertices such that $G\subseteq_{\tau} H^{2\xi}$ via an injection $\tau: V(G)\rightarrow V(H)$,
where every vertex in $H$ is within {\bf wbw}-distance at most $\xi$ from some vertex of $\tau(V(G))$.
\end{lemma}

\begin{figure}[h]
\begin{center}
	\includegraphics[width= 4.9in]{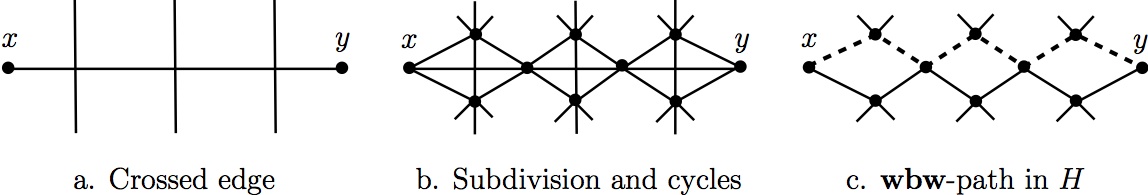}
 \end{center}
 	\caption{\small Edge crossed $k$ times and \wbw-path of length $2k.$}\label{fig:k-cross} 
 \end{figure}
 
\begin{proof}
Let $H$ be the plane transformation of $G$.
Consider an edge $(x,y)\in E(G)$ crossed $\xi'\leq \xi$ times. Notice, for such an edge, it is sufficient to subdivide it with $\xi'-1$ additional vertices.
Then, by construction, in $H$ there is a $(x,y)$-\wbw-path of length $2\xi$ (for an illustration see Figure~\ref{fig:k-cross}). Thus, $G\subseteq H^{2\xi}$. 
Moreover, each vertex on this path has \wbw-distance at most $\xi$ from either $x$ or $y$. 
And since all vertices of $H$ not in $\tau(V(G))$ are created as a subdivision of an edge of $G$,
every vertex in $H$ is within {\bf wbw}-distance at most $\xi$ from some vertex of $\tau(V(G))$.
\end{proof}


\begin{lemma}
\label{triangulate}

Let $H$ be a simple and connected plane graph. Then, there exists 
a simple plane graph $H_{\Delta}$ and an injection $\tau: V(H)\rightarrow V(H_{\Delta})$ such that
\begin{enumerate}
\setlength{\itemsep}{-.0pt}
\item $H_{\Delta}$ is $3$-connected and triangulated,
\item $H\subseteq_{\tau} H_{\Delta}$,
\item for every $x,y\in V(H)$ it holds that $\wbw_{H_{
\Delta}}(x,y)\leq 4\cdot \wbw_{H}(x,y)$,
\item every vertex in $H_{\Delta}$ is within {\bf wbw}-distance at most $4$ from some vertex of $\tau(V(H))$, 
\item for every non-negative integer $\lambda,$ $H^{\lambda}\subseteq H_{\Delta}^{4\lambda}$ (where $H_{\Delta}^{4\lambda}=(H_{\Delta})^{4\lambda}$).
\end{enumerate}
\end{lemma}

\begin{proof}
For every vertex $v$ of $H$, where ${\rm \Sigma}_v$ denotes the cyclic order of its edges,
 we do the following.
If its degree $d(v)$ is not less than two,
we add in the plane $|d(v)|$ vertices and join them to $v$, 
so that between every two consecutive edges in ${\rm \Sigma}_v$ now lies a new edge.
Otherwise, if $d(v)=1$, we add two vertices and join them to $v$ and if 
$d(v)=0$ (i.e. $G$ contains only one vertex),  we add three 
vertices and join them to $v$. 
Next, for every face $f$ in the given embedding of $H$, we add edges between the new vertices that lie now inside $f$, which form a cycle that bounds an empty open disc in the plane.

Note that in the procedure described, all added edges can be drawn so that they don't cross each other or already existing edges. Thus, the resulting graph is again plane. We denote it as $\overline H$.

Notice that we have effectively surrounded each edge of  $H$ with two more edges in $\overline H$, in a way that each pair of vertices in $\overline H$, is connected by at least three disjoint paths. Hence, $\overline H$ is $3$-connected.

Consider now a particular triangulation of $\overline H$, we denote it by $H\tria$,  which is the union of $\overline H$ and its
radial $R_ {\overline H}$, 
i.e. the graph with vertex set $V(R_ {\overline H})$ and edge set $E(H)\cup E(R_ {\overline H})$. As $H$ is a simple graph, all faces of $\overline{H}$ are 
open disks with at least 3 vertices in their boundary. This means that the graph 
$H\tria$ is also 3-connected and therefore it 
satisfies the first two properties of the lemma. 

\begin{figure}
\begin{center}
	\includegraphics[width= 5 in]{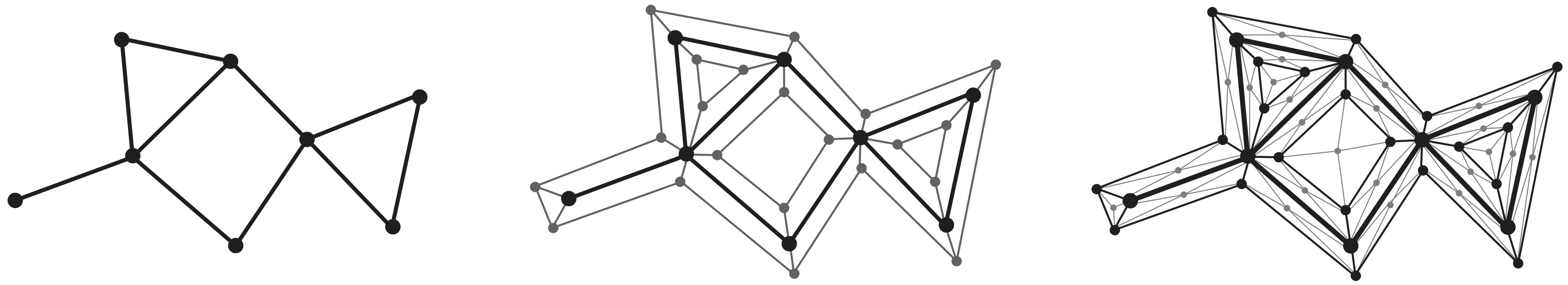}
	 \end{center}
	 \caption{\small An example of the transformation of a graph $H$ to the graphs $\overline{H}$ and $H_{\Delta}$ (in the drawing of  $H_{\Delta}$ we 
	omitted the vertex that corresponds to 
	the  external face of $\overline{H}$ and its incident edges).}\label{fig:ndododsde} 
\end{figure}

\smallskip
Let $x,y$ be arbitrary two vertices in $H$. Consider a $(x,y)$-\wbw-path $P$ in $H$ of length $\lambda'\leq \lambda$. We construct the corresponding \wbw-path $P\tria$ in $H\tria$ where, for every two sequential edges $e_i,e_{i+1}$ of $P$, we follow the steps below. 
The edges $e_i,e_{i+1}$ are attached in $H$. Let $w$ be their common endpoint. By construction, between them lies an edge of $\overline H$, say $\overline e$. Add this edge to the path $P\tria$. Since $\overline H$ is $3$-connected, $(e_i,\overline e)$ and $(\overline e,e_{i+1})$ are incident on two different faces of $\overline H$.
Add the two edges of $R_ {\overline H}$ that have as one endpoint the vertex $w$ and as the other the vertices of $R_ {\overline H}$ in the faces of $\overline H$ incident to $(e_i,\overline e)$ and $(\overline e,e_{i+1})$. Observe, the resulting sequence of edges is a $(x,y)$-\wbw-path in $H\tria$. 
Since we inserted three edges between every two sequential edges of $P$ and the length of $P$ is $\lambda'$, 
the length of the $(x,y)$-\wbw-path $P\tria$ in $H\tria$ is $4 \lambda'-3$ and property 3 holds. 
Note, that this also immediately implies property 5. 

To see that property 4  holds, recall that if a vertex $v$ of $H\tria$ is not in $\tau(V(H))$, 
then either $v\in V(\overline H)\setminus\tau(V(H))$ so $\wbw(v,x)=1$ for a vertex $x \in\tau(V(H))$, 
or $v\in V(R_ {\overline H})$ so $\wbw(v,y)=1$ for a vertex $y \in V(\overline H)$. In the worst case, a \wbw-path of length $4$ in $H\tria$ connects $v$ with a vertex in $\tau(V(H))$.
\end{proof}

\begin{lemma}
\label{flat:cross}
Let $H$ be an $n$-vertex  3-connected triangulated planar graph and $\lambda$ a positive integer. Then, there exists an embedding of $H^{\lambda}$ with at most $2^{\lambda} $ crossings per edge.
\end{lemma}

\begin{proof}
Given a plane graph $H$  and an integer $\lambda\geq 1$, the edges of 
$E(H)$ are called {\em old} edges of $H^{\lambda}$ while the edges in $E(H^\lambda)\setminus E(H)$ are called {\em new} edges of $H^{\lambda}$.

Consider two arbitrary vertices $x,y\in V(H\tria)$ with $\wxy{H\tria}\leq \lambda$, i.e. there exists a
$(x,y)$-\wbw-path $(e_1,e_2,\dots,e_\ell),\ \ell\leq \lambda$. We have to draw a new edge $e$ in $H\tria^\lambda$ joining $x$ and $y$. Let $e_1$ and $e_{\ell}$ be incident
to $x$ and $y$, respectively. Since sequential edges in the \wbw-path are incident to a face, for any $1\leq i\leq \ell-1$, edges $e_i$ and $e_{i+1}$
characterize uniquely the shared face $f_i$. We start drawing $e$ in $f_1$ crossing $e_2$. Then, $e$ goes through $f_2$ crossing $e_3$ and so on till it reaches
$y$ through $f_{\ell-1}$. So, in total $e$ passes through $\ell-1< \lambda$ faces.

Observe that, for every new edge crossing an old edge $e$ and entering face $f$, there are three possibilities: either crossing one of the other two old edges
incident to $f$ or to end in the vertex of the triangle lying opposite to $e$. Since in a triangulated graph all vertices within \wbw-distance of at most 2 are already joined, this yields that  from any vertex $v$ of $H\tria^\lambda$ we can reach with new edges at most $2^{\lambda-2}-1$ other vertices, per each face incident to $v$.
Calculating carefully, we derive that any old edge can be crossed by at most $(\lambda-3)\cdot 2^{\lambda-2}+1$ new edges. 

Now, let us count how many times the new edges can be crossed. 
First we count that the total number of new edges passing through a face of $H\tria$ is at most 
$(3\lambda-6)\cdot 2^{\lambda-3}$. 
Exaggeratively, assume that every new edge is crossed within the face by all other entering edges, i.e. a new edge is crossed inside the face at most $(3\lambda-6)\cdot 2^{\lambda-3}-1$ times. 
Any new edge goes through at most $\lambda-1$ faces and crosses additionally at most
$\lambda-2$ old edges. 
Therefore, any new edge is crossed at most $(\lambda-1)\cdot (3\lambda-6)\cdot 2^{\lambda-3}< 2^{\lambda}$ times.
\end{proof}

Combining the previous lemmata we prove the following corollary, which immediately implies Theorem~\ref{equiv}.

\begin{corollary}
\label{lambda:xi}
For every graph $G$, $\ourlambda(G)/2\leq \ourxi(G)\leq 2^{8\cdot\smallourlambda(G)}$.
\end{corollary}

\begin{proof}
Let $J$ be a plane graph and $\lambda$ an integer such that $G\subseteq J^\lambda$.
W.l.o.g we assume that $J$ is connected (otherwise, we work with each of its connected components separately). Let $H$ be the graph obtained from $J$ after 
subdividing once each of it edges. Observe that $H$ is a simple and connected 
planar graph where  $G\subseteq H^{2\cdot \lambda}$.
Then, we may apply Lemma~\ref{triangulate} and obtain that $G$ is a subgraph of $H\tria^{8\lambda}$, for a planar 3-connected triangulated graph $H\tria$. 
By Lemma~\ref{flat:cross}, we can draw this graph on the plane with at most $2^{8\lambda}$ crossings per edge. We derive that $\ourxi(G)\leq2^{8\cdot \smallourlambda(G)}$. 
Consider now an embedding of $G$ with at most $\xi\leq\ourxi(G)$ crossings per edge. 
By Lemma~\ref{cross:flat}, there is a plane graph $G_1$ such that $G\subseteq G_1^{2\xi}$, i.e., $\ourlambda(G)\leq2\cdot\ourxi(G)$.
\end{proof}

\section{Open problems}
 In this note we introduced the concept of $\lambda$-flat graphs as an alternative 
for $\xi$-nearly planar graphs. We believe that this class has independent graph theoretic interest.
The first question is whether the exponential bound of Corollary~\ref{lambda:xi} 
can become polynomial. Also, it is an open issue whether this new 
concept can be useful to extend known algorithmic results and techniques 
on $\xi$-nearly planar graphs.

%

%

\bibliographystyle{plain}

\end{document}